\documentclass[letterpaper,10pt,conference]{ieeeconf}  %
\IEEEoverridecommandlockouts       
\usepackage{hyperref}
\usepackage{cite}
\usepackage{amsmath,amssymb,amsfonts}
\usepackage{algorithmic}
\usepackage{graphicx}
\usepackage{textcomp}
\usepackage{xcolor}
\usepackage{soul,dsfont} 
\usepackage{tikz} 
\usepackage{color}
\usepackage{algorithm} 
\usepackage{array}
\usepackage{eqparbox}
\usepackage{url}
\usepackage{relsize}
\usepackage{stfloats}
\usepackage{graphics}
\usepackage[mathscr]{euscript}
\graphicspath{{figures/}{./}}

\newtheorem{proposition}{Proposition}

\newtheorem{definition}{Definition}
\newtheorem{corollary}{Corollary}
\newtheorem{remark}{Remark}

\def\BibTeX{{\rm B\kern-.05em{\sc i\kern-.025em b}\kern-.08em
T\kern-.1667em\lower.7ex\hbox{E}\kern-.125emX}}

\title{\LARGE \bf
Stability and Stabilization of Kronecker-Sum-Coupled \\ Multilinear Systems}

\author{Ziqin He, Ruoying Wang, Shenghan Mei, Yidan Mei, Xin Mao, and Can Chen
\thanks{Ziqin He and Ruoying Wang contributed equally to this work.}
\thanks{Ziqin He is with the Department of Mathematics, University of North Carolina at Chapel Hill, Chapel Hill, NC 27599, USA (emails: zhe21@unc.edu).}
\thanks{Ruoying Wang is with the Department of Mathematics, Purdue University, West Lafayette, IN 47907, USA (email: wang6670@purdue.edu).}
\thanks{Shenghan Mei is with the Department of Mathematics, University of North Carolina at Chapel Hill, Chapel Hill, NC 27599, USA (emails: shmei@unc.edu).}
\thanks{Yidan Mei is with the Department of Statistics and Data Science, Yale University, New Haven, CT 06511, USA (emails: yidan.mei@yale.edu).}
\thanks{Xin Mao is with the School of Data and Information Sciences, University of North Carolina at Chapel Hill, Chapel Hill, NC 27599, USA (emails: xinm@unc.edu).}
\thanks{Can Chen is with the School of Data and Information Sciences, the Department of Mathematics, and the Department of Biostatistics, University of North Carolina at Chapel Hill, Chapel Hill, NC 27599, USA (email: canc@unc.edu).}
}

\begin{document}
\maketitle
\begin{abstract}
Coupled dynamical systems with multi-dimensional state structures arise
in many applications such as spatiotemporal processes and large-scale
interconnected systems. However, existing studies primarily focus on
matrix-based coupled systems or tensor-based dynamical systems without
couplings. The properties of coupled tensor-based systems remain largely
unexplored. In this article, we study  Kronecker-sum-coupled multilinear
systems under the transform-based tensor product framework. We establish
fundamental tensor spectral properties showing that stability can be
characterized through lower-dimensional subsystem dynamics together with
slice-wise matrix stability in the transform domain. We further develop
optimization conditions for state-feedback stabilization and show that
the stabilization problem can be decomposed into independent subsystem
problems. Finally, we analyze the computational complexity of the proposed framework and demonstrate its substantial computational advantages over unfolding-based approaches using both synthetic systems and multi-inertia motor systems.
\end{abstract}

\section{Introduction}

Coupled dynamical systems arise naturally in many areas of science and engineering, including networked systems, multi-agent systems, and spatiotemporal processes \cite{mesbahi2010graph,olfati2007consensus,ren2005consensus}. 
In such systems, the global dynamics are often composed of multiple interacting subsystems whose behaviors evolve along different dimensions. 
A classical mathematical representation of such interactions is the Kronecker-sum structure, where the global system matrix is expressed as the sum of Kronecker products of subsystem operators \cite{vanloan2000ubiquitous,khoromskij2018tensor,simoncini2016computational,benner2008numerical}. 
However, traditional approaches typically rely on vectorized
representations of system states.
When the underlying data are naturally multi-dimensional,
such as spatiotemporal signals, tensor-valued measurements, vectorization destroys the
intrinsic multi-linear structure of the data and leads to severe
computational challenges due to the curse of dimensionality.

In recent years, transform-based tensor frameworks have emerged as powerful tools for modeling and analyzing multidimensional dynamical systems \cite{11430563, mei2025controllability,kilmer2011factorization,kilmer2013third,mao2024datadriven,he2026datadrivencontrol,mei2026transform,ding2025tensor,ma2025multilinear}. 
Among these approaches, the tensor-tensor product framework \cite{Ozdemir2022KernelTDA, Kernfeld2015InvertibleTT,lu2018tensor} provides a natural extension of linear algebraic operations to higher-order tensors. 
This framework enables the definition of tensor inverses, eigenvalue decompositions, and singular value decompositions in a manner analogous to classical matrix algebra \cite{braman2010third,chen2021multilinear,kilmer2013third}. 
These developments have facilitated systematic extensions of classical dynamical systems theory to tensor-valued systems, including stability analysis, controllability, and model reduction \cite{chen2024tensorbook,hoover2021new,mei2026tsvd}. 
Nevertheless, most existing works focus primarily on individual multilinear systems governed by a single tensor operator. 
The analysis of coupled multilinear systems remains largely unexplored.

Motivated by these observations, this work studies multilinear systems governed by Kronecker-sum structured operators under the tensor-tensor product framework. 
Such systems evolve along multiple independent modes while still allowing interactions across subsystems. 
By exploiting the transform-domain decoupling property, we develop a systematic framework for analyzing the stability, and stabilization of these coupled multilinear systems. 
We show that the global system behavior can be characterized through lower-dimensional subsystem dynamics and slice-wise matrix problems in the transform domain. 
Based on this, we derive stability conditions and develop linear matrix inequality (LMI)-based criteria for state-feedback stabilization. 
Finally, we analyze the computational complexity of the proposed framework and demonstrate that exploiting tensor structures yields significant computational advantages compared with traditional unfolding-based approaches.


\section{Preliminaries}\label{sec:prelim}

\subsection{Tensor-Tensor Product}

Tensors extend vectors and matrices to multiway data arrays \cite{kolda2009tensor,chen2021multilinear}. 
In this work, we consider $n$th-order tensors 
$\mathscr{A}\in\mathbb{R}^{m_1\times m_2\times \cdots \times m_n}$, 
and introduce a transform-induced tensor algebra based on the tensor-tensor product. 
Specifically, we apply an invertible linear transform along modes $3,4,\ldots,n$, 
carry out matrix multiplication slice-by-slice in the transform domain, 
then map the result back by the inverse transform. To describe these operations, we employ the mode-$k$ tensor-matrix product. 
For a tensor $\mathscr{A}\in\mathbb{C}^{m_1\times m_2 \times \cdots \times m_n}$
and a matrix $\mathbf{M}\in\mathbb{C}^{p\times m_k}$, the mode-$k$ product 
$\mathscr{A}\times_k \mathbf{M}$ produces a tensor in 
$\mathbb{C}^{m_1\times m_2 \times\cdots\times m_{k-1}\times p\times m_{k+1}\times \cdots\times m_n}$ obtained by multiplying 
all mode-$k$ fibers of $\mathscr{A}$ by $\mathbf{M}$
\begin{equation}
(\mathscr{A}\times_k \mathbf{M})_{i_1 i_2\cdots i_{k-1}j i_{k+1}\cdots i_n}
=
\sum_{i_k=1}^{m_k}
\mathscr{A}_{i_1 i_2\cdots i_k\cdots i_n}
\mathbf{M}_{j i_k}.
\end{equation}

This operation enables linear transforms to be applied independently
along individual modes. To define the tensor-tensor product, we first
introduce an invertible transform along modes $3, 4,\ldots,n$.
Let $\mathbf{M}_k\in\mathbb{C}^{m_k\times m_k}$ be invertible for
$k=3, 4, \ldots,n$. Define the forward and inverse transforms by
\begin{align}
\widehat{\mathscr{A}}
&= \mathcal{G}(\mathscr{A})
= \mathscr{A}\times_3 \mathbf{M}_3
\times_4 \mathbf{M}_4 \times_5\cdots \times_n \mathbf{M}_n,
\label{eq:t-transform} \\
\mathscr{A}
&= \mathcal{G}^{-1}(\widehat{\mathscr{A}})
= \widehat{\mathscr{A}}\times_3 \mathbf{M}_3^{-1}
\times_4 \mathbf{M}_4^{-1}\times_5\cdots\times_n \mathbf{M}_n^{-1}.
\end{align}
Here, $\times_k$ denotes the mode-$k$ tensor-matrix product.

\begin{definition}[Tensor-Tensor Product \cite{ozdemir2022high}]
For $n$th-order tensors $\mathscr{A} \in \mathbb{R}^{m_1 \times \ell \times m_3 \times \dots \times m_n}$ and $\mathscr{B} \in \mathbb{R}^{\ell \times m_2 \times m_3 \times \dots \times m_n}$, let $\widehat{\mathscr{A}}$ and 
$\widehat{\mathscr{B}}$ denote their transformed tensors computed by \eqref{eq:t-transform}. 
The tensor-tensor product of $\mathscr{A}$ and $\mathscr{B}$ is defined as
\begin{equation}\label{eq:tproduct}
\mathscr{A} \star \mathscr{B} = \mathcal{G}^{-1} \! \left( \widehat{\mathscr{A}} \odot \widehat{\mathscr{B}} \right) \in \mathbb{C}^{m_1 \times m_2 \times m_3 \times \dots \times m_n},
\end{equation}
where $\odot$ denotes the slicewise matrix product
\begin{equation*}
(\widehat{\mathscr{A}} \odot \widehat{\mathscr{B}})_{::i_3i_4 \cdots i_n} = 
(\widehat{\mathscr{A}})_{::i_3i_4 \cdots i_n} (\widehat{\mathscr{B}})_{::i_3i_4 \cdots i_n}
\end{equation*}
for $i_k = 1,2, \dots, m_k$. The matrices $(\widehat{\mathscr{A}})_{::i_3 i_4\cdots i_n}$ are called frontal slices.
\end{definition}

For convenience, let $i=(i_3,i_4\ldots,i_n)$ denote a multi-index corresponding
to a transformed frontal slice. We define
$
\mathbf{\widehat A^{(i)}} = (\widehat{\mathscr A})_{::i}.
$ Similarly, $\mathbf{\widehat B^{(i)}},\mathbf{\widehat X^{(i)}},\mathbf{\widehat K^{(i)}}$
denote the corresponding frontal slices of
$\widehat{\mathscr B}$, $\widehat{\mathscr X}$, $\widehat{\mathscr K}$.
Several linear algebraic operations can be extended to $n$th-order tensors under the tensor-tensor product. 
The identity tensor
$\mathscr{I}\in\mathbb{R}^{m\times m\times m_3\times\cdots\times m_n}$ 
is defined so that, in the transform domain, each frontal slice of 
$\widehat{\mathscr{I}}$ is the $m\times m$ identity matrix. The transpose of a $n$th-order tensor 
$\mathscr{A}$, 
denoted by 
$\mathscr{A}^{\top}$, 
is obtained by taking the matrix transpose of every transformed frontal slice and mapping the result back to the original domain. 
A tensor is called symmetric if it coincides with its transpose. The inverse of a square tensor 
$\mathscr{A}\in\mathbb{R}^{m\times m\times m_3\times\cdots\times m_n}$ 
is defined by inverting each transformed frontal slice whenever they are nonsingular. 
It satisfies
$
\mathscr{A}\star \mathscr{A}^{-1}
=
\mathscr{A}^{-1}\star \mathscr{A}
=
\mathscr{I}.$ 
And a tensor is positive definite if every transformed frontal slice is positive definite. 
Importantly, matrix eigenvalue decomposition (EVD) can be generalized to $n$th-order tensors.

\begin{definition}[Tensor Eigenvalue Decomposition]
The tensor eigenvalue decomposition (Tensor-EVD) of a square tensor
$\mathscr{A}\in
\mathbb{C}^{m\times m\times m_3\times\cdots\times m_n}$
is defined as
\begin{equation}\label{eq:tEVD}
\mathscr{A}
=
\mathscr{U}
\star
\mathscr{D}
\star
\mathscr{U}^{-1},
\end{equation}
where $\mathscr D$ is f-diagonal, 
$\mathscr{D}_{jj:\cdots:}\in
\mathbb{C}^{m_3\times\cdots\times m_n}$
are called the eigen-tuples.
\end{definition}

Owing to the slice-wise decoupling induced by the tensor-tensor product,
the Tensor-EVD can be computed by performing matrix eigenvalue
decompositions independently on each transformed frontal slice
\begin{align*}
\mathbf{\widehat A^{(i)}}
=
\mathbf{\widehat U^{(i)}}
\mathbf{\widehat D^{(i)}}
(\mathbf{\widehat U^{(i)}})^{-1},
\quad i_k=1,2,\dots,m_k.
\end{align*}
The factor tensors $\mathscr U$ and $\mathscr D$ are obtained by applying
the inverse transform $\mathcal G^{-1}$ to the assembled slice tensors.



\subsection{Tensor Kronecker Product}

Kronecker structures play a central role in structured multilinear systems, 
especially when analyzing separable dynamics and block-interaction models. 

\begin{definition}[Tensor Kronecker Product \cite{pickard2024kronecker}]
Consider
$\mathscr{A}\in\mathbb{C}^{m_1\times m_2\times\cdots\times m_n}$
and
$\mathscr{B}\in\mathbb{C}^{p_1\times p_2\times\cdots\times p_n}$.
The tensor Kronecker product
$\mathscr{A}\otimes\mathscr{B}
\in\mathbb{C}^{(m_1p_1)\times(m_2p_2)\times\cdots\times(m_np_n)}$
is defined by
\begin{align*}
(\mathscr{A}\otimes\mathscr{B})_{\alpha_1\alpha_2\cdots\alpha_n}
=
\mathscr{A}_{i_1i_2\cdots i_n}\,
\mathscr{B}_{j_1j_2\cdots j_n},
\end{align*}
where $\alpha_k=i_k+(j_k-1)m_k$, $1\le i_k\le m_k$, $1\le j_k\le p_k$, for $k=1,2,\dots,n$.
\end{definition}

Let $\mathbf{M}_k^{(A)}$ and $\mathbf{M}_k^{(B)}$ be invertible transform matrices for $\mathscr{A}$ and $\mathscr{B},$ respectively. Define the Kronecker-lifted transform matrices
$
\mathbf{M}_k = \mathbf{M}_k^{(B)} \otimes \mathbf{M}_k^{(A)} \in \mathbb{C}^{(p_k m_k)\times(p_k m_k)}
$, $k=3,4,\ldots,n$.
We write the forward transform of $\mathscr{A}\otimes\mathscr{B}$ by
$$
\widehat{\mathscr{A}\otimes\mathscr{B}}
=
(\mathscr{A}\otimes\mathscr{B})
\times_3 \mathbf{M}_3
\times_4 \mathbf{M}_4\times_5
\cdots
\times_n \mathbf{M}_n.
$$

\begin{proposition}[Tensor Kronecker  Property]
\label{prop:kron-transform}
Let $\mathscr A$ and $\mathscr B$ be tensors of compatible sizes. Then the transformed
frontal slices of their tensor Kronecker product satisfy
\begin{equation}
\bigl(\widehat{\mathbf A\otimes\mathbf B}\bigr)^{(\boldsymbol{\alpha})}
=
\widehat{\mathbf A}^{(\mathbf i)}
\otimes
\widehat{\mathbf B}^{(\mathbf j)},
\label{eq:kron-transform}
\end{equation}
 where $\boldsymbol{\alpha}$ is the corresponding slice index defined by
$
\alpha_k=i_k+(j_k-1)m_k$ for $k=3,4,\dots,n$.
\end{proposition}

\begin{proof}
Let
$\widehat{\mathscr A}=\mathcal G_A(\mathscr A)$ and
$\widehat{\mathscr B}=\mathcal G_B(\mathscr B)$. For the tensor Kronecker product, the transform along mode $k$ is
induced by the Kronecker-lifted matrix
$\mathbf M_k^{(B)}\otimes\mathbf M_k^{(A)}$. Hence,
\[
(\mathscr A\otimes\mathscr B)\times_k
\bigl(\mathbf M_k^{(B)}\otimes\mathbf M_k^{(A)}\bigr)
=
(\mathscr A\times_k\mathbf M_k^{(A)})
\otimes
(\mathscr B\times_k\mathbf M_k^{(B)}).
\]
Applying this identity successively for $k=3,4,\dots,n$ gives
$
\widehat{\mathscr A\otimes\mathscr B}
=
\widehat{\mathscr A}\otimes\widehat{\mathscr B}.
$
Therefore, each transformed frontal slice satisfies
$
\bigl(\widehat{\mathbf A\otimes\mathbf B}\bigr)^{(\boldsymbol{\alpha})}
=
\widehat{\mathbf A}^{(\mathbf i)}
\otimes
\widehat{\mathbf B}^{(\mathbf j)}$. 
\end{proof}

\section{Stability}
We next study the stability of continuous-time multilinear systems
with Kronecker-structured operators. The global operator admits an
additive Kronecker structure that combines lower-dimensional dynamics
across different modes. We consider
multilinear systems governed by a tensor Kronecker-sum operator 
\cite{vanloan2000ubiquitous, benner2008numerical} of the form
\begin{equation}
\mathscr{A}\oplus\mathscr{B}
=
\mathscr{A}\otimes\mathscr{I}
+
\mathscr{I}\otimes\mathscr{B},
\end{equation}
where $\mathscr{I}$ is the identity tensor, $\mathscr{A},\mathscr{B}\in
\mathbb{C}^{m\times m\times m_3\times\cdots\times m_n}$. The resulting continuous-time
dynamics are given by
\begin{equation}\label{eq:cts_tensor_system}
\dot{\mathscr{X}}(t)
=
(\mathscr{A}\oplus\mathscr{B})
\star
\mathscr{X}(t).
\end{equation}
where $\star$ denotes the tensor-tensor product.

\subsection{Stability of Continuous-Time Multilinear Systems}

We consider the continuous-time  multilinear system in the form of \eqref{eq:cts_tensor_system}, where $(\mathscr{A}\oplus\mathscr{B})\in\mathbb{C}^{m^2\times m^2\times m_3^2\times\cdots\times m_n^2}$
.

\begin{definition}[Stability]\label{def:Lstability}
Let $\mathscr{A}\oplus\mathscr{B}$ admit the Tensor-EVD
$
\mathscr{A}\oplus\mathscr{B}
=
\mathscr{U}\star\mathscr{D}\star\mathscr{U}^{-1}.
$
The system \eqref{eq:cts_tensor_system} is said to be
\begin{itemize}
\item stable if all entries of eigen-tuples of $\mathscr{A}\oplus\mathscr{B}$ have non-positive real parts;
\item asymptotically stable if all entries of eigen-tuples have negative real parts;

\item unstable if at least one entry of the eigen-tuples has positive real part.
\end{itemize}
\end{definition}

Since the Tensor-EVD assumes a diagonalizable representation, defective
eigenvalues on the imaginary axis are excluded. The stability of the multilinear system can be analyzed conveniently in the transform domain by reducing the dynamics to a collection of independent matrix systems.

\begin{corollary}[Slice-Wise Stability Characterization]
\label{cor:slicewise_stability} 
The continuous-time multilinear system \eqref{eq:cts_tensor_system} 
is (asymptotically) stable if and only if each transformed frontal slice is a (asymptotically) stable matrix.
\end{corollary}

\begin{proof}
Under the transform $\mathcal{G}$, the tensor-tensor product 
decouples into independent matrix systems across transformed 
frontal slices. Stability therefore follows directly from classical 
matrix stability theory applied slice-wise.
\end{proof}

\subsection{Stability of  Kronecker-Sum-Coupled Systems}

We now analyze the stability of multilinear systems generated by
Kronecker-sum operators under the tensor-tensor product framework, in the form of \eqref{eq:cts_tensor_system}. To expose the spectral structure of these systems,
we relate the eigen-tuples of $\mathscr{A}\oplus\mathscr{B}$
to those of the factors.



\begin{proposition}
[Eigen-Tuples of Tensor Kronecker-Sum]
\label{cor:eigtuple-kronsum}
Let $\mathscr{A}$ and $\mathscr{B}$ admit Tensor-EVDs. 
Let $\lambda_r^A(\mathbf i)$ and $\lambda_q^B(\mathbf j)$ denote
the eigenvalues of the transformed slices
$\widehat{\mathbf A}^{(\mathbf i)}$ and
$\widehat{\mathbf B}^{(\mathbf j)}$.
Then the eigenvalues of the corresponding transformed slice
$\boldsymbol{\alpha}$ of $\mathscr A\oplus\mathscr B$ satisfy
\begin{align}
\lambda_{r,q}^{A\oplus B}(\boldsymbol{\alpha})
=
\lambda_r^A(\mathbf i)
+
\lambda_q^B(\mathbf j),
\label{eq:eigtuple-kronsum}
\end{align}
where $\boldsymbol{\alpha}$ is the combined transformed slice index
corresponding to $(\mathbf i,\mathbf j)$ as defined above.
Consequently, the eigen-tuples of $\mathscr A\oplus\mathscr B$
are obtained by collecting these pairwise sums over all transformed
slice indices.
\end{proposition}

\begin{proof}
Let $\widehat{\mathscr A}=\mathcal G(\mathscr A)$ and
$\widehat{\mathscr B}=\mathcal G(\mathscr B)$.
By linearity of $\mathcal G$ and the definition of the tensor Kronecker-sum,
\begin{align*}
\widehat{\mathscr A\oplus\mathscr B}
=
\widehat{\mathscr A\otimes\mathscr I}
+
\widehat{\mathscr I\otimes\mathscr B}.
\end{align*}
Using Proposition~\ref{prop:kron-transform}, each transformed frontal
slice satisfies
\[
\bigl(\widehat{\mathbf A\oplus\mathbf B}\bigr)^{(\boldsymbol{\alpha})}
=
\widehat{\mathbf A}^{(\mathbf i)}\otimes \mathbf I
+
\mathbf I\otimes\widehat{\mathbf B}^{(\mathbf j)},
\]
where $\boldsymbol{\alpha}$ is the combined transformed slice index
corresponding to $(\mathbf i,\mathbf j)$, and $\mathbf I$ denotes the
identity matrix of compatible size. The right-hand side is the matrix
Kronecker-sum
$
\widehat{\mathbf A}^{(\mathbf i)}
\oplus
\widehat{\mathbf B}^{(\mathbf j)}.
$
Therefore, by the classical spectral property of the matrix
Kronecker-sum, its eigenvalues are the pairwise sums of the eigenvalues
of $\widehat{\mathbf A}^{(\mathbf i)}$ and
$\widehat{\mathbf B}^{(\mathbf j)}$. Hence,
$
\lambda_{r,q}^{A\oplus B}(\boldsymbol{\alpha})
=
\lambda_r^A(\mathbf i)
+
\lambda_q^B(\mathbf j).
$
Collecting these eigenvalues over all corresponding transformed slice
indices gives the eigen-tuples of $\mathscr A\oplus\mathscr B$.
\end{proof}

\begin{corollary}[Stability Criterion]
\label{cor:spectral-sum}
If both $\mathscr{A}$ and $\mathscr{B}$ generate asymptotically stable
multilinear systems,   the Kronecker-sum-coupled multilinear system generated by
$\mathscr{A}\oplus\mathscr{B}$ is asymptotically stable.
\end{corollary}

\begin{proof}
Asymptotic stability implies that all eigen-tuples of
$\mathscr{A}$ and $\mathscr{B}$ have negative real parts.
And
the eigen-tuple of $\mathscr{A}\oplus\mathscr{B}$ is sum of two such
eigen-tuple elements.
The conclusion follows from Definition~\ref{def:Lstability}.
\end{proof}

\begin{remark}[Computational Complexity]
Denote the number of transformed frontal slices of each of
$\mathscr{A}$ and $\mathscr{B}$ by
$
s=m_3m_4\cdots m_n.
$
The direct unfolding of
$\mathscr{A}\oplus\mathscr{B}$ has dimension
$(m^2s^2)\times(m^2s^2)$, and directly computing its eigenvalues
requires
$
O\bigl((m^2s^2)^3\bigr)
$
operations. In contrast, by Proposition~\ref{cor:eigtuple-kronsum}, the
tensor-based approach first computes the EVD
of the $s$ transformed frontal slices of each of $\mathscr{A}$ and
$\mathscr{B}$,
requiring $O(m^3s)$ operations. The eigenvalues of
$\mathscr{A}\oplus\mathscr{B}$ are then obtained by forming pairwise
sums between the eigenvalues of
$\widehat{\mathbf{A}}^{(\mathbf{i})}$ and
$\widehat{\mathbf{B}}^{(\mathbf{j})}$ over all $s^2$ slice pairs with
$O(m^2s^2)$ operations. Therefore, the overall computational
complexity of the tensor-based approach is
$
O(m^3s+m^2s^2)
$.
\end{remark}

\section{Stabilization by State Feedback}

We next study the stabilization of continuous-time multilinear
systems whose system operator has a Kronecker-sum structure under the
tensor-tensor product framework. Consider the controlled multilinear system
\begin{equation*}
\dot{\mathscr X}(t)
=
(\mathscr A_1\oplus\mathscr A_2)\star\mathscr X(t)
+
(\mathscr B_1\otimes\mathscr I)\star\mathscr U_1(t)
+
(\mathscr I\otimes\mathscr B_2)\star\mathscr U_2(t),
\end{equation*}
where
$
\mathscr A_1\oplus\mathscr A_2
=
\mathscr A_1\otimes\mathscr I
+
\mathscr I\otimes\mathscr A_2.
$
Here, $\mathscr X(t)$ denotes the state tensor, while
$\mathscr U_1(t)$ and $\mathscr U_2(t)$ denote control inputs acting
along the first and second Kronecker factors, respectively.

To preserve the separable structure, we consider the
mode-separable state-feedback laws
$
\mathscr U_1(t)
=
(\mathscr K_1\otimes\mathscr I)\star\mathscr X(t)$, $
\mathscr U_2(t)
=
(\mathscr I\otimes\mathscr K_2)\star\mathscr X(t),$
where $\mathscr K_1$ and $\mathscr K_2$ are designed independently
from the corresponding subsystem stabilization problems.
The objective is to choose these tensors such that the closed-loop multi-linear system is asymptotically stable.

\subsection{State-Feedback Stabilization of Multilinear Systems}

We first recall the general state-feedback stabilization problem for
continuous-time multilinear systems.
\begin{equation}
\dot{\mathscr X}(t)
=
\mathscr A\star\mathscr X(t)
+
\mathscr B\star\mathscr U(t),
\label{eq:general_controlled_tensor_system}
\end{equation}
where
$\mathscr X(t)\in
\mathbb C^{m\times r\times m_3\times\cdots\times m_n}$,
$\mathscr U(t)$ is the control input tensor, and
$\mathscr A,\mathscr B$ are compatible system tensors. Under the state-feedback law
$
\mathscr U(t)=\mathscr K\star\mathscr X(t),
$
the controlled multi-linear system becomes the closed-loop system
\begin{align*}
\dot{\mathscr X}(t)
=
(\mathscr A+\mathscr B\star\mathscr K)\star\mathscr X(t).
\end{align*}
The tensor $\mathscr K$ is referred to as a
state-feedback gain tensor.
The stabilization problem consists of designing
$\mathscr K$ such that the closed-loop multilinear system
is asymptotically stable.

\begin{definition}[Stabilizing State Feedback]
Feedback tensor $\mathscr K$ is called stabilizing if the closed-loop multilinear system is asymptotically stable.
\end{definition}

\begin{proposition}[LMI-Based Stabilization]
\label{prop:lmi_stabilization}
Consider the multilinear system \eqref{eq:general_controlled_tensor_system}.
If there exist tensors $\mathscr Q\succ0$ and $\mathscr Y$ such that
$$
\mathscr A\star\mathscr Q
+
\mathscr Q\star\mathscr A^*
+
\mathscr B\star\mathscr Y
+
\mathscr Y^*\star\mathscr B^*
\prec0,
$$
then the state-feedback gain
$
\mathscr K=\mathscr Y\star\mathscr Q^{-1}
$
asymptotically stabilizes the multi-linear system.
\end{proposition}

\begin{proof}
Under the state-feedback law
$
\mathscr U=\mathscr K\star\mathscr X,
$
the closed-loop system becomes
$
\dot{\mathscr X}
=
(\mathscr A+\mathscr B\star\mathscr K)\star\mathscr X.
$
Let
$
\mathscr A_{\rm cl}=\mathscr A+\mathscr B\star\mathscr K.
$
By the Lyapunov criterion, the closed-loop system is asymptotically
stable if there exists a tensor
$\mathscr P=\mathscr P^*\succ0$ such that
$
\mathscr A_{\rm cl}^*\star\mathscr P
+
\mathscr P\star\mathscr A_{\rm cl}
\prec0.
$
Let
$
\mathscr Q=\mathscr P^{-1}.
$
Since $\mathscr P\succ0$, we also have $\mathscr Q\succ0$.
Left- and right-multiplying the above inequality by $\mathscr Q$ gives
$
\mathscr Q\star\mathscr A_{\rm cl}^*
+
\mathscr A_{\rm cl}\star\mathscr Q
\prec0.
$
Substituting
$
\mathscr A_{\rm cl}=\mathscr A+\mathscr B\star\mathscr K
$
and defining
$
\mathscr Y=\mathscr K\star\mathscr Q,
$
the inequality becomes
$
\mathscr A\star\mathscr Q
+
\mathscr Q\star\mathscr A^*
+
\mathscr B\star\mathscr Y
+
\mathscr Y^*\star\mathscr B^*
\prec0.
$
Therefore, if there exist tensors $\mathscr Q\succ0$ and $\mathscr Y$
satisfying the above LMI, then defining
$
\mathscr K=\mathscr Y\star\mathscr Q^{-1}
$
yields a feedback gain for which the closed-loop system satisfies the
strict Lyapunov inequality and is therefore asymptotically stable.
\end{proof}

The state-feedback design problem can be analyzed conveniently in the transform domain.
Let
$
\widehat{\mathscr A}=\mathcal G(\mathscr A)$, $
\widehat{\mathscr B}=\mathcal G(\mathscr B)$, $
\widehat{\mathscr K}=\mathcal G(\mathscr K).
$
Since the tensor-tensor product decouples across transformed frontal slices, the closed-loop multilinear system is equivalent to a collection of independent matrix systems.

\begin{corollary}[Slice-Wise Stabilization Characterization]
\label{cor:slicewise_stabilization}
A feedback tensor $\mathscr K$ is stabilizing for the multilinear system if and only if,
for every transformed frontal slice, the matrix gain $(\widehat{\mathbf K})^{(\mathbf{i})}$ stabilizes the corresponding slice system.
\end{corollary}

\begin{proof}
Applying the transform $\mathcal G$ to the closed-loop dynamics yields
$
\widehat{\dot{\mathscr X}}
=
(\widehat{\mathscr A+\mathscr B\star\mathscr K} )\star\mathscr X.
$
By linearity of $\mathcal G$ and the definition of the tensor-tensor product,
multiplication in the transform domain becomes slice-wise matrix multiplication.
Consequently, the multilinear system decouples across transformed
frontal slices. For each slice index
$i=(i_3,i_4,\dots,i_n)$, the dynamics reduce to the matrix system
\begin{align*}
\dot{\mathbf{X}}^{(\mathbf{i})}
=
\Big((\widehat{\mathbf A})^{(\mathbf{i})}
+
(\widehat{\mathbf B})^{(\mathbf{i})}(\widehat{\mathbf K})^{(\mathbf{i})}\Big)\mathbf{X}^{(\mathbf{i})}.
\end{align*}
Thus, the closed-loop system is asymptotically stable
if and only if every slice system is asymptotically stable.
\end{proof}

\subsection{Stabilization of  Kronecker-Sum-Coupled Systems}

We consider multilinear systems whose system operator
has a Kronecker-sum structure. The controlled multilinear system is
\begin{align*}
\dot{\mathscr X}(t)
=
&(\mathscr A_1\oplus\mathscr A_2)\star\mathscr X(t)
+
(\mathscr B_1\otimes\mathscr I)\star\mathscr U_1(t)\\
&+
(\mathscr I\otimes\mathscr B_2)\star\mathscr U_2(t),
\end{align*}
where
$
\mathscr A_1\oplus\mathscr A_2
=
\mathscr A_1\otimes\mathscr I
+
\mathscr I\otimes\mathscr A_2.
$ We solve the stabilization problems for the two
constituent subsystems independently. These subsystem feedback operators are then combined
in a mode-separable form that preserves the Kronecker-sum structure of
the closed-loop multilinear system.

\begin{proposition}[Mode-Separable Stabilization]
\label{cor:kronsum_stabilization}
Suppose the feedback tensors $\mathscr K_1$ and $\mathscr K_2$ are
obtained independently by solving the stabilization problems for the
two constituent subsystems, such that
$\mathscr A_1+\mathscr B_1\star\mathscr K_1$
and
$\mathscr A_2+\mathscr B_2\star\mathscr K_2$
are asymptotically stable. Consider the mode-separable state-feedback laws
$
\mathscr U_1(t)
=
(\mathscr K_1\otimes\mathscr I)\star\mathscr X(t)$, $
\mathscr U_2(t)
=
(\mathscr I\otimes\mathscr K_2)\star\mathscr X(t).
$
Then the resulting closed-loop multi-linear system  
\begin{align*}
\dot{\mathscr X}(t)
=
\Big(
(\mathscr A_1+\mathscr B_1\star\mathscr K_1)
\oplus
(\mathscr A_2+\mathscr B_2\star\mathscr K_2)
\Big)\star\mathscr X(t).
\end{align*}
is  asymptotically stable.
\end{proposition}

\begin{proof}
Substituting
$
\mathscr U_1(t)
=
(\mathscr K_1\otimes\mathscr I)\star\mathscr X(t)
$
and
$
\mathscr U_2(t)
=
(\mathscr I\otimes\mathscr K_2)\star\mathscr X(t)
$
into the controlled multilinear system and using associativity of the
tensor-tensor product gives
\begin{align*}
\dot{\mathscr X}(t)
=
&\Big(
\mathscr A_1\oplus\mathscr A_2
+
(\mathscr B_1\otimes\mathscr I)
\star(\mathscr K_1\otimes\mathscr I)\\
&+
(\mathscr I\otimes\mathscr B_2)
\star(\mathscr I\otimes\mathscr K_2)
\Big)\star\mathscr X(t).
\end{align*}
By the mixed-product property,
$
(\mathscr B_1\otimes\mathscr I)
\star(\mathscr K_1\otimes\mathscr I)
=
(\mathscr B_1\star\mathscr K_1)\otimes\mathscr I
$
and
$
(\mathscr I\otimes\mathscr B_2)
\star(\mathscr I\otimes\mathscr K_2)
=
\mathscr I\otimes(\mathscr B_2\star\mathscr K_2).
$
Hence, by the definition of the  Kronecker-sum,
\begin{align*}
&\mathscr A_1\oplus\mathscr A_2
+
(\mathscr B_1\star\mathscr K_1)\otimes\mathscr I
+
\mathscr I\otimes(\mathscr B_2\star\mathscr K_2)
\\
&=
(\mathscr A_1+\mathscr B_1\star\mathscr K_1)\otimes\mathscr I
+
\mathscr I\otimes
(\mathscr A_2+\mathscr B_2\star\mathscr K_2)
\\
&=
(\mathscr A_1+\mathscr B_1\star\mathscr K_1)
\oplus
(\mathscr A_2+\mathscr B_2\star\mathscr K_2).
\end{align*}
Therefore, the closed-loop multilinear system is
\begin{align*}
\dot{\mathscr X}(t)
=
\Big(
(\mathscr A_1+\mathscr B_1\star\mathscr K_1)
\oplus
(\mathscr A_2+\mathscr B_2\star\mathscr K_2)
\Big)\star\mathscr X(t).
\end{align*}
By assumption, the two constituent closed-loop operators
$\mathscr A_1+\mathscr B_1\star\mathscr K_1$ and
$\mathscr A_2+\mathscr B_2\star\mathscr K_2$
are asymptotically stable. 
By Proposition~\ref{cor:eigtuple-kronsum},
the eigen-tuples of their Kronecker-sum are obtained from sums of
the eigen-tuple elements of the two constituent closed-loop
operators. Since all such elements have negative real parts,
their sums also have negative real parts. Therefore, the
resulting closed-loop multilinear system is asymptotically stable.
\end{proof}

\begin{remark}[Computational Complexity]
Denote the number of transformed frontal slices of each constituent
system by $s=m_3m_4\cdots m_d$. Under the full tensor Kronecker
product, the Kronecker-sum-coupled system has $s^2$ transformed
frontal slices, each of size $m^2\times m^2$. Hence, direct
vectorization yields a system of dimension $N=m^2s^2$, and solving
the corresponding LMI requires $O((m^2s^2)^6)$ operations.
In contrast, Proposition~\ref{cor:kronsum_stabilization} decomposes
the stabilization problem into $2s$ independent LMIs of size
$m\times m$. Thus, the tensor-based approach has complexity
$O(sm^6)$, leading to substantial computational savings.
\end{remark}
\section{Numerical Examples}\label{sec:num}

In this section, we illustrate the computational advantages of the proposed tensor-based framework for Kronecker-sum multilinear systems. 
All experiments were conducted in MATLAB on a machine equipped with an M1 Pro CPU and 16GB of memory. The code is available at
\url{https://github.com/ZiqinHe/Tensor_Kronekcer-sum}.

\subsection{Stability}

We compare the computational time of the proposed tensor method with the
direct unfolding method for the stability analysis of Kronecker-sum-coupled
multi-linear systems. In the first experiment, we fix the number of
transformed slices to $R=4$ and vary the matrix dimension $m$.
The tensor method computes the eigenvalues of
$\widehat{\mathbf A}_1^{(i)}$ and $\widehat{\mathbf A}_2^{(\ell)}$,
$i,\ell=1,2,\dots,R$, separately and obtains the eigenvalues of the
$R^2$ Kronecker-sum slices by pairwise sums, while the unfolding method
directly analyzes the full matrix representation.
As shown in Fig.~\ref{fig:complexity}a, the tensor method scales much
more favorably with $m$. In the second experiment, we fix $m$ and vary the number of transformed
slices $R$. As shown in Fig.~\ref{fig:complexity}b, the tensor method
also exhibits substantially better scalability with respect to $R$.
\begin{figure}[t]
\centering
\includegraphics[width=0.49\textwidth]{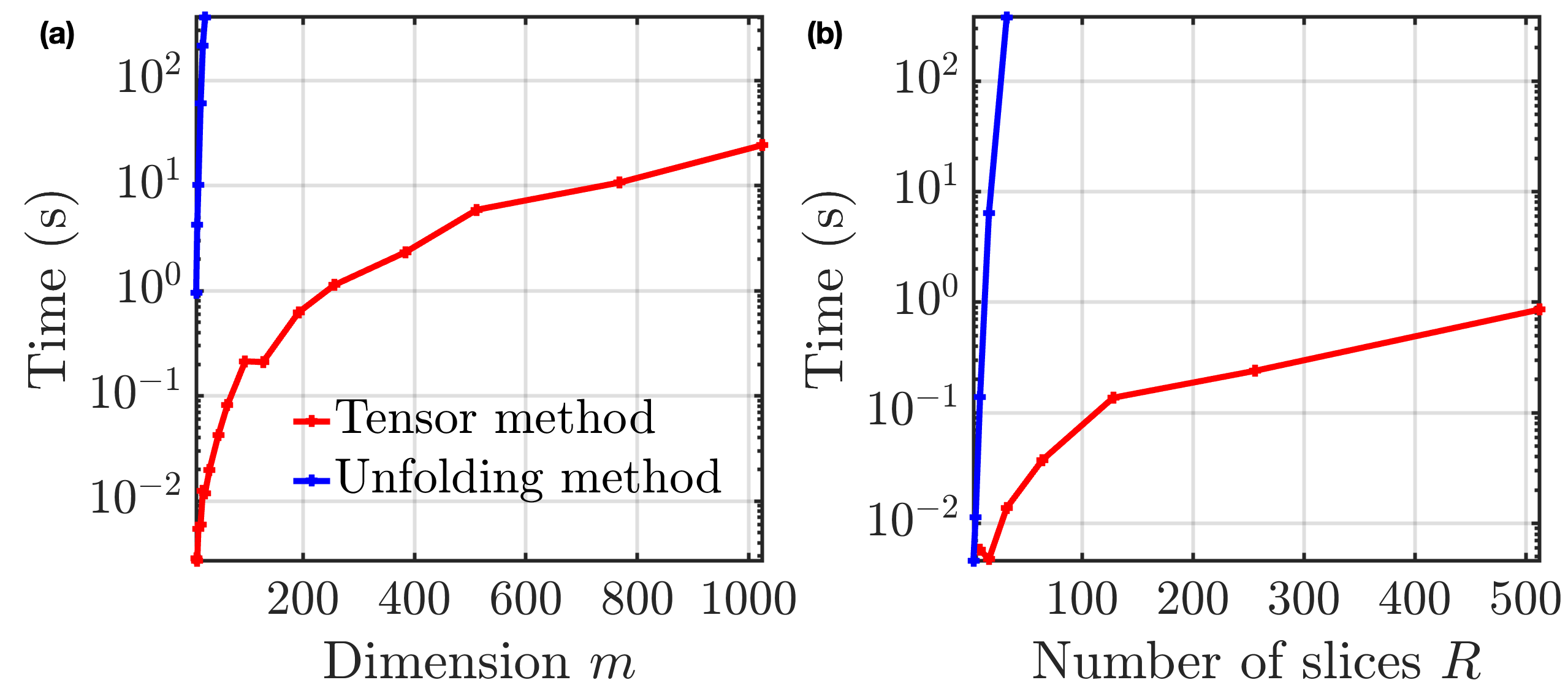}
\vspace{-15pt}
\caption{Computational time comparison between the tensor method and the direct unfolding method for the stability analysis of tensor Kronecker-sum systems. 
(a) Runtime versus the matrix dimension with a fixed number of transformed slices. 
(b) Runtime versus the number of transformed slices. }
\label{fig:complexity}
\end{figure}

\subsection{Stabilization}

\begin{figure}[t]
\centering
\includegraphics[width=0.49\textwidth]{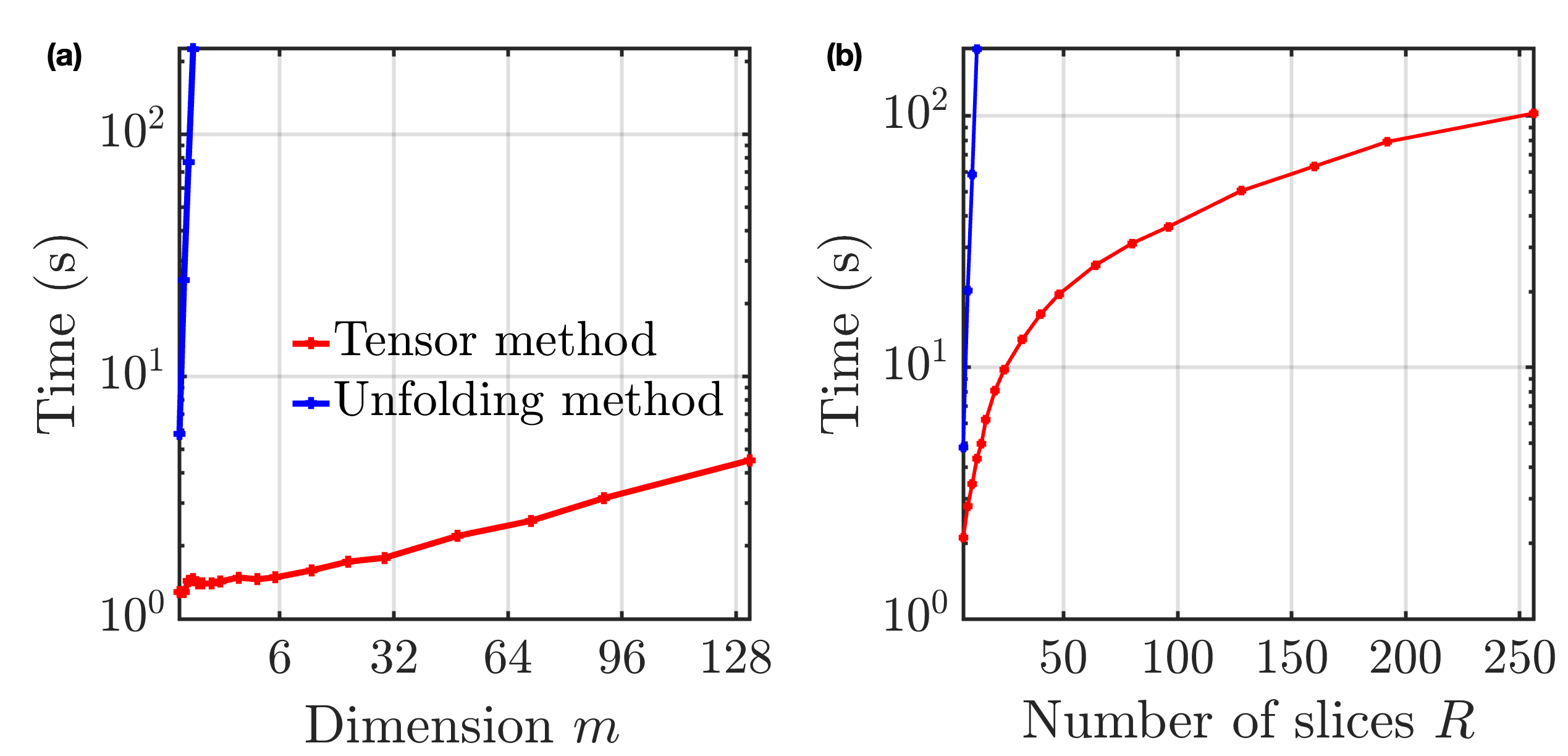}
\vspace{-15pt}
\caption{Computational time comparison for stabilization of tensor Kronecker-sum systems using tensor method and unfolding method. 
(a) Runtime versus the state dimension $m$ with a fixed number of slices $R=8$. 
(b) Runtime versus the number of transformed slices with fixed matrix dimension $m=3$.}
\label{fig:stabilization}
\vspace{-10pt}
\end{figure}

We compare the computational time of the proposed tensor-based stabilization
method with the direct unfolding method for Kronecker-sum-coupled multi-linear
systems. In the first experiment, we fix the number of transformed slices to
$R=2$ and vary the matrix dimension $m$. The tensor method exploits the
separable Kronecker-sum structure and solves the constituent stabilization
problems independently, whereas the unfolding method constructs a single LMI
of size $R^2m^2$. As shown in Fig.~\ref{fig:stabilization}a, the unfolding
method becomes rapidly more expensive as $m$ increases due to the high
computational complexity of semidefinite programming and its large memory
requirements, while the tensor method scales much more favorably. In the second experiment, we fix $m$ and vary the number of transformed
slices $R$. As shown in Fig.~\ref{fig:stabilization}b, the increasing LMI
size causes the computational cost of the unfolding method to grow rapidly
with $R$, whereas the tensor-based method remains feasible for substantially
larger problems.

\subsection{Tensor Multi-Inertia Motor Systems}
\label{subsec:motor-stability}

Multi-inertia motor-drive systems are commonly used to model flexible
electromechanical systems, including servo drives, robotic joints,
multi-axis motion systems, and mechanical transmission systems. In such
systems, multiple inertias are connected through flexible shafts,
torsional springs, and dampers, making stability and stabilization
important considerations in system analysis and controller design. We therefore consider a  Kronecker-sum
model constructed from two multi-inertia motor dynamics as an
engineering example of our framework.

Let $\mathscr{E}_i$ and $\mathscr{A}_i$, $i=1,2$, denote the tensor
descriptor and system operators.
Each motor factor contains $q$ coupled inertias and has subsystem
dimension $m=2q$, corresponding to the angular positions and angular
velocities. The constituent tensor operators are constructed in the form
\begin{equation*}
\mathscr{E}_i
=
\begin{bmatrix}
\mathscr{I}_q & \mathscr{O}\\
\mathscr{O} & \mathscr{M}_i
\end{bmatrix},
\
\mathscr{A}_i
=
\begin{bmatrix}
\mathscr{O} & \mathscr{I}_q\\
-\mathscr{K}_i & -\mathscr{D}_i
\end{bmatrix},
\qquad i=1,2,
\label{eq:tensor-motor-factors}
\end{equation*}
where $\mathscr{M}_i$, $\mathscr{K}_i$, and $\mathscr{D}_i$
represent the tensor-valued inertia, stiffness, and damping parameters,
respectively. Neighboring inertias are coupled through torsional springs
and dampers, while variations along the tensor modes represent different
operating conditions or parameter configurations.
Using these two motor factors, we construct the original tensor
Kronecker-sum descriptor system
$$
\mathscr{E}_{\otimes_t} \star \dot{\mathscr{X}}(t)
=
\mathscr{A}_{\oplus} \star \mathscr{X}(t),$$
where $
\mathscr{E}_{\otimes_t}
=
\mathscr{E}_1\otimes\mathscr{E}_2, 
\mathscr{A}_{\oplus}
=
\mathscr{A}_1\otimes\mathscr{E}_2
+
\mathscr{E}_1\otimes\mathscr{A}_2.$

\begin{figure}[t]
    \centering
    \includegraphics[width=0.49\textwidth]{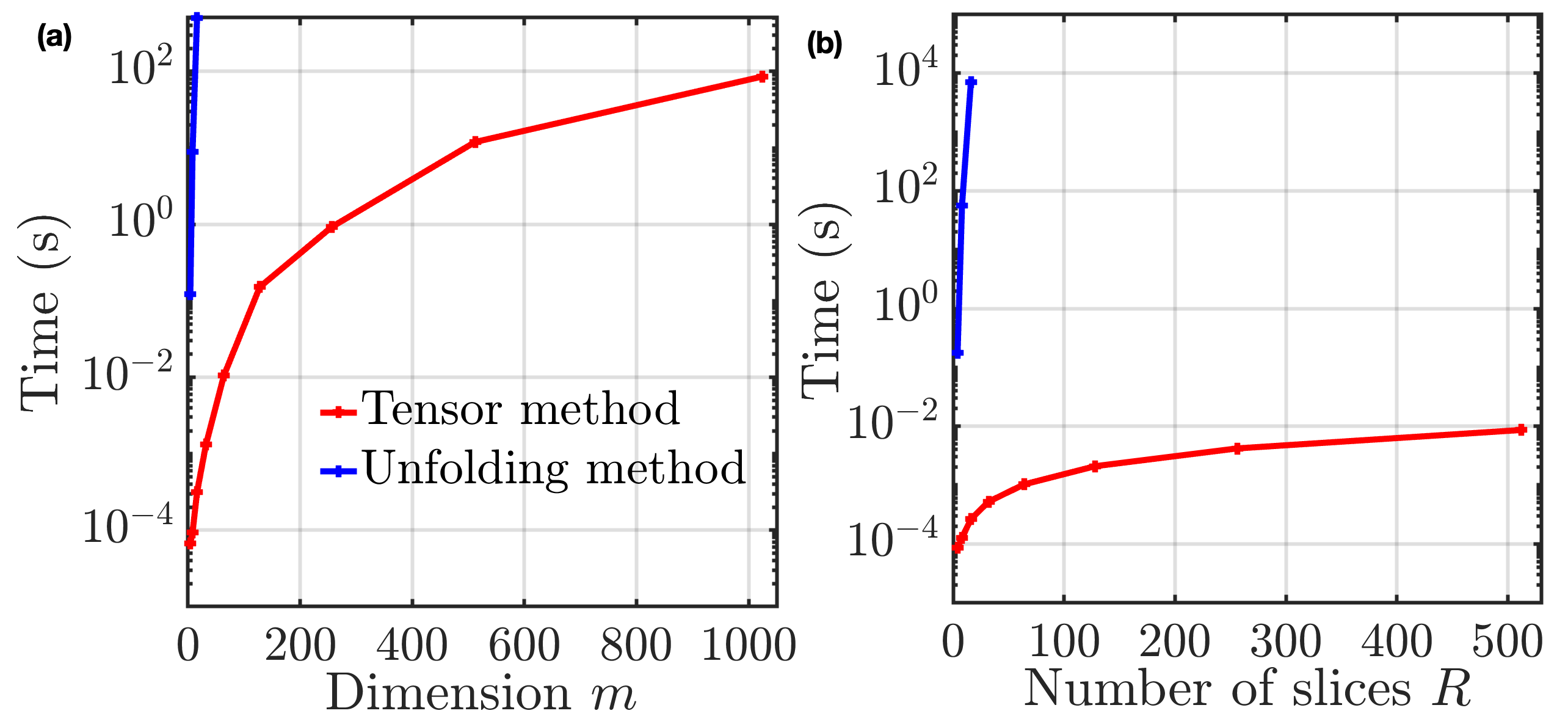}
    \vspace{-15pt}
    \caption{Computational time for stability analysis of the tensor
    multi-inertia motor Kronecker-sum system.
    (a) Runtime versus the subsystem dimension.
    (b) Runtime versus the number of transformed slices.}
    \label{fig:motor-stability}
    \vspace{-10pt}
\end{figure}

\paragraph{Stability}
We compare the tensor-tensor and unfolding methods by varying the subsystem
dimension $m$ with a fixed number of transformed slices and, conversely,
varying the number of slices $R$ with a fixed subsystem dimension.
The two approaches produce spectra that agree to numerical precision.
As shown in Fig.~\ref{fig:motor-stability}, the computational cost of the
unfolding method increases rapidly with both $m$ and $R$ due to the rapidly
growing dimension of the unfolded system, whereas the tensor-tensor method
exploits the Kronecker-sum structure and remains computationally efficient
for substantially larger systems.

\begin{figure}[t]
    \centering
    \includegraphics[width=1\linewidth]{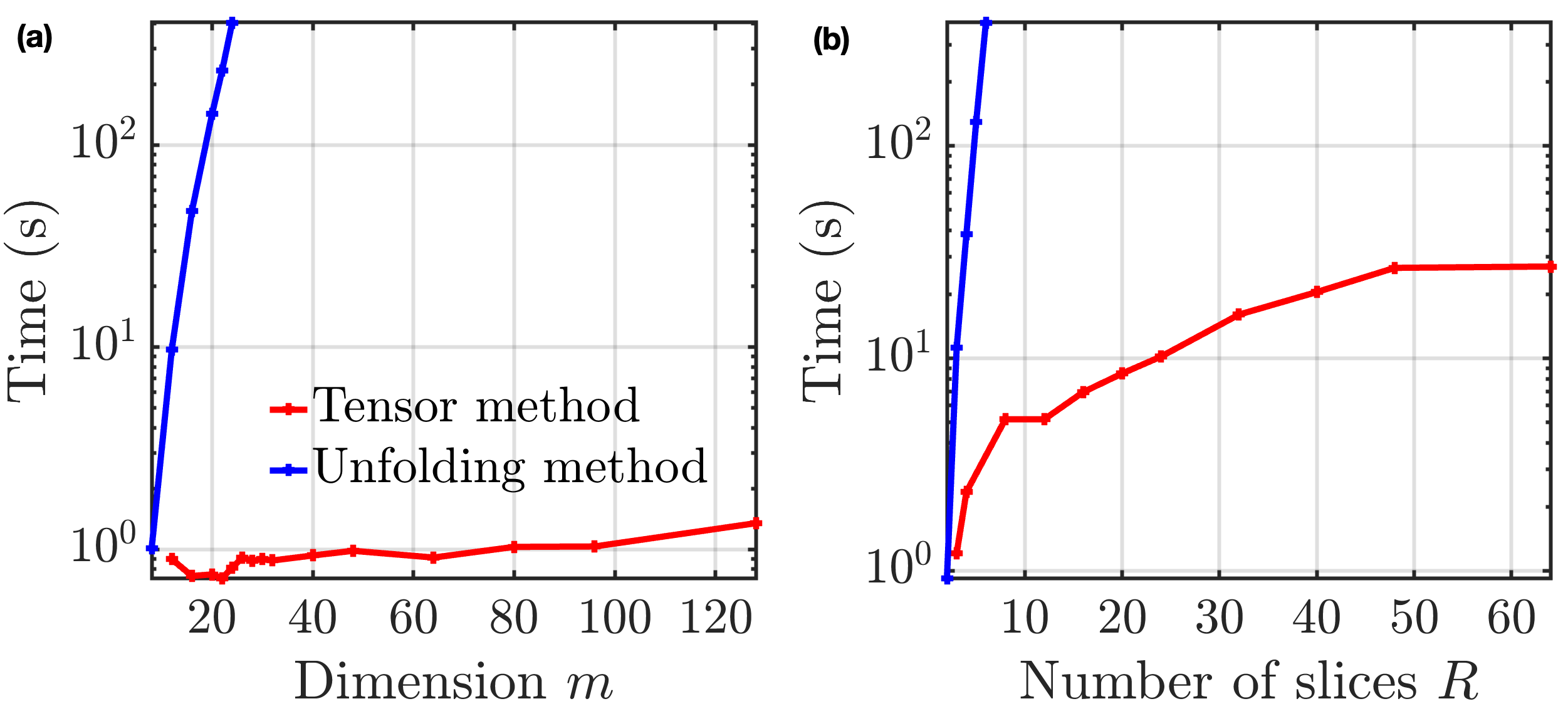}
    \vspace{-15pt}
    \caption{Computational time for stabilization of the tensor
    multi-inertia motor Kronecker-sum system.
    (a) Runtime versus the subsystem dimension.
    (b) Runtime versus the number of transformed slices.}
    \label{fig:motor-stabilization}
    \vspace{-10pt}
\end{figure}

\paragraph{Stabilization.}
We next consider stabilization of the same Kronecker-sum motor system by
introducing a small shift to generate unstable open-loop dynamics.
Both approaches successfully stabilize all tested systems; for example,
at $m=30$ and $R=2$, the spectral abscissa is reduced from
$2.377\times10^{-1}>0$ to $-1<0$ by the tensor-tensor feedback.
As shown in Fig.~\ref{fig:motor-stabilization}, the tensor-tensor method
exhibits substantially better computational scaling. For example, at
$m=20$, it requires about $7.5\times10^{-1}$ s, compared with
$1.43\times10^{2}$ s for unfolding. Similarly, as the number of slices
$R$ increases, the unfolding cost grows rapidly, while the tensor-tensor
method remains feasible for substantially larger systems.

\section{Conclusion}\label{sec:con}
In this article, we studied stability analysis and stabilization of  Kronecker-sum-coupled multilinear systems under the tensor-tensor product framework. 
By exploiting both the subsystems decomposition induced by the Kronecker-sum structure and the slice-wise decoupling property in the transform domain, the proposed formulation reduces large-scale tensor problems to a collection of smaller subsystem and slice-wise matrix problems, enabling efficient eigenvalue computation and LMI-based stabilization analysis. Future work includes extending the proposed framework to nonlinear tensor-based systems.
For such systems, the proposed transform-domain representation may be
used to exploit the slice-wise structure, while the
nonlinearity may introduce coupling among the transformed slices.
Developing stability and stabilization conditions for such
nonlinear cross-slice interactions is an interesting direction for
future research.

\bibliographystyle{IEEEtran}
\bibliography{reference.bib}
\end{document}